\documentclass[11pt]{amsart}
\usepackage[active]{srcltx}

\usepackage{latexsym}

\usepackage{amsmath,amssymb,amsthm,amsfonts,latexsym}
\usepackage{pstricks, pst-node, pst-text, pst-3d}
\usepackage[bookmarks]{hyperref}

\def\cal{\mathcal}

\let\=\partial

\newtheorem {pro}{Proposition}[section]
\newtheorem {thm}[pro]{Theorem}
\newtheorem{lem}[pro]{Lemma}

\theoremstyle{definition}
 \newtheorem {rem}[pro]{Remark}
\newtheorem {dfn}[pro]{Definition}

\newcommand{\Q} {\mathbb{Q}}

\newcommand{\R}{\mathbb{R}}

\newcommand{\N}{\mathbb{N}}

\newcommand{\rc}{{R}}
\newcommand{\V}{\mathcal{V}}

\newcommand{\hn}{\mathcal{H}}

\newcommand{\St}{\mathcal{S}}

\newcommand{\ep}{\varepsilon}

\newcommand{\pa}{\partial}

\title {A generalized Sard theorem on real closed fields}



\author[A. Valette and G. Valette]{Anna Valette and Guillaume Valette}

\address[A. Valette]{Instytut Matematyki Uniwersytetu
Jagiello\'nskiego, ul. S \L ojasiewicza, Krak\'ow, Poland}
\email{anna.valette@im.uj.edu.pl}
\address[G. Valette]
{Instytut Matematyczny PAN, ul. \'Sw. Tomasza 30, 31-027 Krak\'ow,
Poland} \email{gvalette@impan.pl}

\keywords{Semi-algebraic mappings, asymptotic critical values}

\thanks{Research partially supported by the NCN grant}

\subjclass[2010]{58K05, 14P10,  57R35}

\begin{document}
\maketitle

\begin{abstract}
We work with semi-algebraic functions on arbitrary real closed fields. 
We generalize the notion of critical values and prove a Sard type theorem in our framework. 
\end{abstract}

\section{Introduction}
The famous Sard's Theorem asserts that the set of critical values of a ${\cal C}^\infty$ map $f:\R^n\to\R^k$ is of Lebesgue measure zero. 
Let us recall that the set of critical values of $f$ is the image under $f$ of all the elements at which $f$ fails to be a submersion. 
In this article, we focus on semi-algebraic functions on real closed fields. 
The notion of Lebesgue measure does not really make sense on a real closed field.
However, the dimension of a semi-algebraic set can always be defined 
and it is well known that one can prove that the set of critical values of a mapping $f:R^n \to R^k$ has dimension less than $k$ \cite{bcr}, 
if $R$ stands for a real closed field. In this note, we establish a related result. 
We introduce a more general notion of critical values and show a Sard type theorem for these critical values, 
estimating the size of the set of our generalized critical values. There is no nice geometric measure theory on real closed fields. 
We therefore make use of the notion of $v$-thin sets introduced in \cite{v}. Further features of these critical values will be developped in forthcoming articles.

 We then give an application of our Sard's theorem to the study of asymptotic critical values of real mappings $f:X \to \R^k$, $X\subset \R^n$. We shall show that the set of asymptotic critical values of a semi-algebraic mapping $f:X \to \R^k$ has dimension less than $k$.  A similar result was also obtained by K. Kurdyka, P. Orro and S. Simon \cite{kos} who showed that the set of  generalized critical values at infinity for semi-algebraic mappings has dimension less than $k$. This result lead them to establish a generalization of Ehresmann's fibration theorem for nonproper mappings. This provided an effective way to ensure stability and yield that the set of bifurcation points of a semi-algebraic mapping is relatively small.



By $\rc$ we denote a real closed field. We refer to \cite{bcr} for basic facts about semi-algebraic sets. We denote by $cl(X)$ and $int(X)$ the closure and interior of a set and set $\delta X:=cl(X)\setminus int(X)$ as well as $fr(X):=cl(X)\setminus X$.  We will write $|x|$ for the Euclidean norm of $x\in R^n$ and $d(\cdot, \cdot)$ for the Euclidean distance, while $B^n(x,r)$ will stand for the ball of radius $r$ centered at $x$ for this metric.

\section{$v$-thin sets}We introduce the notion of $v$-thin sets as it was defined in \cite{v}.

\begin{dfn}\label{z-thin-set}
Let $X\subset\rc^n$  be a semi-algebraic subset and let $z\in\rc$ be positive.  A $k$-dimensional semi-algebraic subset
$X$ of $R^n$ is called {\bf $z$-thin} if for a generic orthogonal projection $\pi:\rc^n\to H^k$, where $H^k$ is a $k$-dimensional vector subspace of $R^n$, the image $\pi(X)$ does not contain any open ball of radius $z$.


\end{dfn}

Let $v\subset \rc$ be a convex subgroup, i.e. a set such that for any $x,y\in v$ and for any $t\in [0,1]\subset\rc$ we have $tx+(1-t)y\in v$.

 For example, suppose that  the field $R$ is the field of real algebraic Puiseux series $k(0_+)$ (on a right-hand-side neighborhood of $0$ in $\R$) endowed with the order that makes the indeterminate $T$ positive and smaller than any positive real number. Then, one can choose $v$ as the group of all the series satisfying $|x| \leq N T^2$ for  any $N$ in  $\R$. In the Hardy field of $ln-exp$ definable germs of one variable functions
 (in a right-hand side neighborhood of zero)
one may consider  the set of all the $L^p$ integrable germs of
series.

\begin{dfn}\label{thin-set}
We say that a set $X$ is {\bf $v$-thin} if it is $z$-thin for some positive $z\in v$.  

Let $j\ge \dim X $ be an integer.  For simplicity we say that $X$ is {\bf$(j,v)$-thin} if  it is   $v$-thin or if  $\dim X<j$ (with the convention $\dim \emptyset:=-1$). 

\end{dfn}

We shall need the fact that a semi-algebraic family of $v$-thin sets is $v$-thin (Proposition \ref{family}). 
This requires two preliminary lemmas.  Given $z \in R$ and $Y\subset R^n$, we define {\bf the $z$-neighborhood of $Y$} as the set $$\V_z(Y):=\{ x \in R^n| d(x,Y)<z \}.$$

\begin{lem}\label{lem_z_nghbd}
 If $Y\subset R^n$ is a semi-algebraic subset of dimension less than $n$, then for any $z \in v$, the set  $\V_z(Y)$ is $v$-thin.
\end{lem}
\begin{proof}Let $Y$ be of dimension $<n$ and   $z \in v$ be positive. We shall indeed show  that $\V_z(Y)$ is $N z$-thin for some $N \in \N$.
Since $\dim Y< n$ we can decompose the set $Y$ into semi-algebraic sets $A_1\cup\dots\cup A_k$ such that there exist some orthogonal automorphisms $\varphi_i:R^n\to R^n$, $i=1,\dots , k$,  such that $\Gamma_i:=\varphi_i(A_i)$ is the graph of a semi-algebraic $C$-Lipschitz function $\xi_i:B_i\to \R$, with $C\in \N$, $B_i \subset \R^{n-1}$ (see \cite{f}).

 We will proceed by induction on $k$, the case $k=0$ being vacuous. For simplicity, we will assume $\varphi_k=Id_{R^n}$ (we can work up to an orthogonal map).  Assume the result true for $(k-1)$, $k\ge 1$, and take $M \ge 4(C+1)$, $M\in \R$. 

Let $a \in A_k$. We claim that there is $b\in B^n(a,M^2 \cdot z)$, such that  $B^n(b,M\cdot z)\cap  A_k=\emptyset$.  Assume otherwise. It means that $d(b,A_k) \le M z$, for any $b\in B^n(a,M^2\cdot z)$. Take $b\in B^n(a,M^2 \cdot z)$ such that $|b-a|\ge \frac{M^2}{2}z$ and  $\pi(b)=\pi(a)$, where $\pi:R^n\to R^{n-1}$ is the canonical projection. Then, as $A_k$ is the graph of a $C$-Lipschitz function, a straightforward computation gives (since $M\ge \frac{C+1}{4}$): 
 $$|b-a|\le (C+1)d(b,A_k)\le (C+1)Mz\le \frac{M^2}{4}z,$$
a contradiction.

By induction on $k$, if we set $Y':= A_1\cup \dots \cup A_{k-1}$ then $\V_z(Y')$ must be $N z$-thin, for some $N\in\N$.   Hence, for any $b \in \V_z(Y')$, the ball $B^n(b,N\cdot z)$ does not fit in $\V_z(Y')$. Let $z':=2M^2 z$, where $M:=N+1$. We are going to show that $B^n(a,z')$ does not fit in $\V_z(Y)$, which will yield the desired result. By the above paragraph (we may assume $N\ge 4(C+1)$), we know that there is $b\in B^n(a,M^2\cdot z)$  such that $B^n(b,M\cdot z)$ does not meet $A_k$. Hence, $B^n(b, N\cdot z)$ does not meet  $\V_z(A_k)$.  As (by choice of $N$) this ball is not included in $\V_z(Y')$, it therefore cannot be included in $\V_z(Y)=\V_z(Y')\cup \V_z(A_k)$.  
As $B^n(b, N\cdot z)\subset B^n(a,z')$, we are done.
\end{proof}


Let us recall that given two nonempty subsets $X$ and $Y$ of a semialgebraic subset of $ R^n$, the Hausdorff distance is defined as:
$$d_\hn(X,Y):= \max \;(\sup_{x \in X} \inf_{y \in Y} d(x,y), \, \sup_{y\in Y} \inf_{x\in X} d(x,y)).$$
Here, we adopt the convention that $d_\hn(X,Y)$ is equal to infinity if one of these suprema is infinite.

 Given $X \in R^{m+n}$ and $t \in R^m$ we set
 $$X_t:=\{x \in R^n |(x,t)\in  X\}.$$
Any family $(X_t)_{t \in R^m}$ constructed in this way is then called {\bf a semi-algebraic family of sets}.

\begin{lem}\label{thin}
 Let $X\subset\rc^n$ be a semi-algebraic  set of dimension $k$. The set $X$ is $v$-thin iff there exists a semi-algebraic set $Y\subset R^n$ of dimension less than $k$ such that
 $d_{\cal H}(X,Y)\in v$.  Moreover, if $(X_t)_{t \in R^m}$ is a semi-algebraic family of subsets of $R^n$, then the corresponding family  $(Y_t)_{t \in R^m}$ can be chosen semi-algebraic as well.
\end{lem}
\begin{proof}
 We start with the ``only if'' part. If  the set $X$ has nonempty interior, it is enough to take  $Y$ as the set $\delta X$. Indeed, if  $d_{\cal H}(X,Y)\notin v$, it follows that there exists a ball of radius $z\in v$ which fits in $X$, which implies that $X$ is not $v$-thin.  In the case where  $X$ is of empty interior, we can decompose  this set into a finite union of graphs of Lipschitz mappings and the result follows from the case $\dim X=n$.

Let us establish the ``if'' part. By Lemma \ref{lem_z_nghbd}, the projection of $X$ onto a generic $k$-dimensional hyperplane lies in a $z$-neighborhood of $Y$  for $z\in v$, which, by the previous Lemma, is a $v$-thin set.  Hence, $X$ is $v$-thin.

 By construction of $Y$ and quantifier elimination, if $(X_t)_{t \in R^m}$ is a semi-algebraic family of subsets of $R^n$, then the corresponding family  $(Y_t)_{t \in R^m}$ is semi-algebraic as well (since $\delta X_t$ is then a semi-algebraic family).
\end{proof}

As a matter of fact, we get the following proposition which will be useful to establish our Sard type theorem.

\begin{pro}\label{family}
Let $X\subset R^{m+n}$ be a semi-algebraic set. If $X_t$ is  $v$-thin for all $t\in R^m$ then $X$ is a $v$-thin set.
\end{pro}
\begin{proof}
Observe that by Lemma \ref{thin} for any $t \in R^m$ we have a semi-algebraic set $Y_t$ of smaller dimension and $z\in v$ such that $X_t\subset \V_z(Y_t)$. Clearly, the set $$Y=\bigcup\limits_{t\in R^m} Y_t\times\{t\}$$ has dimension less than $\dim X$. As $X \subset \V_z(Y)$, by Lemma \ref{thin}, this shows that the set $X$ is $v$-thin.    
\end{proof}

\section{Sard's theorem}
Let us recall the definition of the Rabier function \cite{r} which measures the distance of a linear operator $A$ to the set of singular operators: $$\nu(A):=\inf_{||\varphi||=1}||A^\ast\varphi||.$$
 See \cite{kos,r} for basic properties of this function. In the case where $A$ is an operator from $R^n$ to $R^k$ with $n\ge k$, one can equivalently characterize $\nu(A)$ by $$\nu(A)=\inf\{|Av|| v \perp \ker A, |v|=1 \}.$$

 We say that $z\in R$ is an {\bf infinitesimal} if $|z|$ is smaller than any positive rational number. 
We will say that a subset of $R^n$ is {\bf  bounded} if it is included in a ball $B^n(0,N)$ in $R^n$, with $N \in \N$.   The field $R$ being not necessarily Archimedean it is worthy of notice that we require $N$ to be an  integer.

\begin{dfn}
Let $X\subset\rc^n$ be a bounded manifold, $z\in\rc$ be an infinitesimal,  and  
let $f:X\to {\rc}^k$ be a ${\cal C}^1$ semi-algebraic mapping. 
We say that  $x\in X$ is a {\bf $z$-critical point of $f$}  if $\nu(d_xf)$ is less than $z$.
We denote the set of all $z$-critical points by $c_z(f)$, i.e. $$c_z(f)=\{x\in X|\, \nu(d_xf)<z\}.$$
\end{dfn}

We now come to the main theorem of this article.

\begin{thm}\label{sard}Let $v$ be a convex subgroup of $\rc$  and  
let $f:X\to {\rc}^k$ be a ${\cal C}^1$ semi-algebraic function, $X$ bounded manifold. 
For any  infinitesimal $z\in v$ the image $f(c_z(f))$ is $(k,v)$-thin.
\end{thm}
\begin{proof}
 We shall argue by induction on $k$. 
 We may assume that $f$ is everywhere a submersion, since, by the classical Sard's theorem \cite{bcr}, the image of the singular points has dimension less than $k$.

Assume first that $k=1$ and take a partition $\Sigma$ of  $c_z(f)$ such that for every $S\in \Sigma$, every couple of points $x,y$ of $S$ can be joint by a piecewise ${\cal C}^1$ arc $\gamma:[0,1]\to  S$ that has derivative bounded by $N \in \N$ (that such a partition exists follows from existence of $L$-regular cell decompositions \cite{f}). Take two arbitrary points $x$ and $y$ in $S\in \Sigma$ as well as such an arc $\gamma$. Let $\phi(t):=f(\gamma(t))$ and observe that $$|\phi'(t)|=|<\pa_{\gamma(t)}f, \gamma'(t)>|\le Nz$$ which means in particular that $\phi'(t)$ belongs to $v$.  
  By the mean value theorem, there is $ t \in [0, 1]$ such that $$\phi(1)-\phi(0) =\phi'(t).$$
As $|\phi'(t)|\le  Nz$ for all $t \in [0,1]$, we get that   $|f(x)-f(y)|\le  N z$. This shows that $f(c_z(f))$ is $(1,v)$-thin.

 We now assume the result true for $(k-1)$, $k \ge 2$.  Let $\Gamma$ denote the graph of $f_{|c_z(f)}$. 
   Let $n$ be such that the set $X$ is included in $R^n$ and let $\pi:R^{n+k}\to R^k$ be the projection onto the  $k$ last coordinates. For each $x\in \Gamma$ let $w(x)$ be a unit vector such that 
\begin{equation}|\pi(w(x))|=\sup\{|\pi(w)|\,\,|\,\, |w|=1\,\, \mbox{\rm and}\,\, w\in T_x\Gamma\}.\end{equation}
By semi-algebraic choice, the vector field $w$ can be required to be a semi-algebraic.  The vector $w(x)$ is normal to 
$$K_x:=\ker \pi_{|T_x\Gamma}=T_x\Gamma\cap (R^{n}\times\{0_{R^k}\}).$$
Let $v(x)$ be  a semi-algebraic vector field such that 
\begin{equation}
|\pi(v(x))|=\min\{|\pi(v)|\,\,|\,\, |v|=1, v\perp K_x, v\in T_x\Gamma\}.
\end{equation}
Let 
$$ b(x):=|\pi(w(x))|.$$ Since $k\ge 2$, the vector $v$  can be required to be ortogonal to $w$. 
As $w$ is a semi-algebraic unit vector field, for any positive rational number $\varepsilon$, there is a finite partition of $\Gamma$, say $V_1,\dots, V_p$, such that for any $i$, there is $f_i\in S^{k-1}$ such that for each $x\in V_i$ we have:
$$
|\frac{\pi(w(x))}{|\pi(w(x))|}-f_i|<\varepsilon.$$
We can work separately on every $V_i$  and up to a linear isometry of $R^k$. We will thus assume for simplicity that for all $x\in \Gamma$ we have \begin{equation}\label{1}|e_{k}-\frac{\pi(w(x))}{|\pi(w(x))|}|<\varepsilon ,\end{equation} where $e_{k}$ is the last vector of canonical basis of $R^k$.  We do not specify $\ep$ as it simply needs to be a small positive constant in $\Q$.
Let also $$\Delta_x:=e_{n+k}^{\perp}\cap T_x\Gamma.$$
We claim that 
\begin{equation}\label{3}
d(w(x),\Delta_x)\ge 1-\varepsilon>0.
\end{equation} 
Assume otherwise i.e. $w(x)=w_1(x)+w_2(x)$ with $w_1(x)\in\Delta_x$ and $|w_2(x)|< 1-\varepsilon $. 
By definition of $w$ and $b$ we get
\begin{equation}\label{2}
|<e_{k},\pi(w(x))>| \overset{(\ref{1})}{\ge}  b(x) (1-\varepsilon).\end{equation}
Moreover, $$|\pi(w_2(x))|\le b(x)|w_2(x)|< b(x)(1-\varepsilon).$$
As $<e_{k},\pi(w_1)>\equiv 0$ we get 
$$|<e_{k},\pi(w(x))>|=|<e_{k},\pi(w_2(x))>|<b(x)(1-\varepsilon),$$ contradicting (\ref{2}).

Since $v(x)$ is orthogonal to $w(x)$ we thus have by (\ref{3}) for $\ep$ small enough:
\begin{equation}\label{omega}
 d(v(x),\Delta_x)\le\frac{1}{2}.
\end{equation}
 We can decompose the vector $v$ as follows:
$$v=v_1+\alpha w,\; \,\, \mbox{ with}\, v_1\in\Delta_x, \;\alpha\in\R.$$
As $w$ and $v$ are normal to $K_x$, so is $v_1$.
Let $\pi_2:R^{k}\to e_{k}^\perp $  be the projection along the span $<\pi(w)>$. Observe that it follows from (\ref{1}) that $|\pi_2| \le 2$ (provided $\ep$ was chosen small enough). Note also that, by (\ref{3}) and (\ref{omega}), we have $\alpha< \frac{3}{4}$ (for $\ep$ small enough),  so that \begin{equation}\label{eq_v1}
|v_1(x)|\ge 1-\frac{3}{4}= \frac{1}{4}.\end{equation}

%
Hence, as $\pi(v_1(x))\in e_{k}^\perp$ (since $v_1(x)\in e_{n+k}^\perp$) , we have for any $x\in \Gamma $:
\begin{equation}\label{eq_pi_v_1}
|\pi(v_1(x))|=|\pi_2(\pi(v_1(x)))|=|\pi_2(\pi(v(x)))| \le 2|\pi(v(x))|\le 2z. 
 \end{equation}

It means that for every $t \in R$, every point $x\in  \Gamma_t$ is an $8z$-critical point of $\pi_{|\Gamma_t}$ (since $v_1(x)\perp K_x$ and, by (\ref{eq_v1}), $|v_1(x)|\ge \frac{1}{4}$). 
Let $$A_t:=\Lambda( \Gamma_t),$$ where $\Lambda:R^n \times R^k \to R^n$ is the projection onto the first factor. Clearly, $$\dim A_t=\dim \Gamma_t\le l-1$$ (the set $\Gamma_t$ cannot be $l$-dimensional for $f$ is a submersion). For $t\in R$,  write $f(x,t)$ as $(g_t(x),t)\in R^{k-1}\times R$ for each   $x\in A_t$.  This defines a mapping $g_t:A_t \to R^{k-1}$ (observe that the graph of $g_t$ is $\Gamma_t$).

 We claim that every point of $A_t$ is a $16z$-critical for  $g_t$, i.e. that  $c_{16z}(g_t)=A_t$. To see this, pick a point $\xi \in A_t$. As $x=(\xi,g_t(\xi))$ is  $8z$-critical point of $\pi_{|\Gamma_t}$, we know that there is a unit vector $\omega\in T_x \Gamma_t$ such that $|\pi(\omega)| \le 8z$. This implies that $$|\Lambda(\omega)|\ge \sqrt{1-(8 z)^2 }\ge \frac{1}{2}$$ (since $z$ is an infinitesimal). Hence, if we set $\omega':= \frac{\Lambda(\omega)}{|\Lambda(\omega)|}$ we get  $$|d_x g_t(\omega')|\le 2| \pi(\omega)|\le 16 z,$$ 
as required. 

 Consequently, by induction on $k$, the set  $g_t(A_t)$ is $v$-thin for every $t\in R$.  As, by definition of $g$,  $g_t(A_t)=f(c_z(f))_t$, by Lemma \ref{family}, this entails that the set $f(c_z(f))$ is $v$-thin.
\end{proof}

\section{Application to the field of the real numbers}
Let $\St(0,\ep)$ denote the ring of semi-algebraic functions on $(0,\ep)$ (in $\R$), $\ep $ positive real number, 
and let   $k(0_+)$ be the ordered field defined by the direct limit 
$\underset{\longrightarrow}{lim}\, \St((0,\ep))$. 
It is the set of semi-algebraic functions $f:(0,\ep)\to \R$, quotiented by the equivalence relation that identifies two semi-algebraic functions that coincide on a right-hand-side neighborhood of zero.
 Let then $v_{0_+}$ be the additive subgroup of $k(0_+)$ consisting of the germs at $0$ of functions of $\St(0,\ep)$  tending to zero at $0$. It is nothing but the set of all the infinitesimals of $k(0_+)$.

 Let $(A_t)_{t \in \R}$ be a semi-algebraic family of subsets of $\R^n$ and denote by $A_{0_+}$ the generic fiber of the family, that is to say the subset of $k(0_+)^n$ constituted by all the germs of map-germs $f\in k(0_+)^n$ such that $f(t) \in A_t$ for all $t$ positive and small enough (see \cite{bcr} for more details on generic fibers). If $(g_t)_{t \in \R}$ is a semi-algebraic family of functions,   we then denote by $g_{0_+}$ the generic fiber of the family $g_t$, i.e. the mapping whose graph is the generic fiber of the family constituted by the graphs of the functions $g_t$, $t \in \R$.

\begin{lem}\label{lem_familles}
 If  $A_{0_+}$ is $v_{0_+}$-thin then the set $$B:=cl( \bigcup_{t>0} (A_t \times \{t\}))\cap (\R^n \times \{0\})$$ has dimension less than $\dim A_t$, for $t >0$ small ($\dim A_t$ is constant for $t>0$ small).  
\end{lem}
\begin{proof}
If $A_{0_+}$ is $v_{0_+}$-thin then by  Lemma \ref{thin} there is a subset $Y_{0_+}$ of $k(0_+)^n$ such that $d_{\hn}(A_{0_+},Y_{0_+})$ is an infinitesimal and $\dim Y_{0_+}<\dim A_{0_+}$. It means that the set $Y_{0_+}$ is the generic fiber of a family $(Y_t)_{t \in \R}$ such that $d_{\hn}(A_{t},Y_{t})$ tends to zero as $t$ positive goes to zero. This shows that 
$$cl( \bigcup_{t\in \R} (A_t \times \{t\}))\cap (\R^n \times \{0\})=cl( \bigcup_{t\in \R} (Y_t \times \{t\}))\cap (\R^n \times \{0\}),$$
which is of dimension less than $\dim A_{0_+}$. 
%
\end{proof}


\subsection{Asymptotic critical values.}
Let $f:X \to \R^k$ be a semi-algebraic mapping, with $X$ (not necessarily bounded) manifold. 
Recall that the set of critical values of $f$ is
$$K_0(f):=\{y\in \R^k\;|\; \exists x\in f^{-1}(y)\, , \nu(d_x f)=0\}.$$
\begin{dfn}\label{dfn_critical_values}
The set 
$$K(f):=\{y\in\R^k\;|\; \exists  x_n \in X\, ,  f(x_n)\to y\, \mbox{\rm and}\, (1+|x_n|)\nu(d_{x_n}f)\to 0\}.$$
is called the {\bf set of generalized critical values}.
\end{dfn}

This set of course contains all the critical values of $f$ but also other elements that are sometimes called {\bf asymptotic critical values}. In particular, $K(f)$ contains {\bf the asymptotic critical values at infinity} studied in \cite{kos} and defined as:
$$K_{\infty}(f):=\{y\in\R^k\;|\; \exists x_n \in X\, , |x_n|\to+\infty\, , f(x_n)\to y,\, \mbox{\rm and}\, |x_n|\nu(d_{x_n}f)\to 0\}.$$
We also set $$K_1(f):=\{y\in\R^k\;|\; \exists x_n\in X\, , |x_n|\to fr( X),  f(x_n)\to y,\, \mbox{\rm and}\, \nu(d_{x_n}f)\to 0\}.$$ 

These are the asymptotical critical values at the points of $fr( X)$. Clearly, $$K(f)=K_0(f)\cup K_1(f)\cup K_\infty(f).$$

It is proved in  \cite{kos} that the set $K_\infty(f)$ is of dimension less than $k$.  Theorem \ref{sard} entails  a Sard Theorem for asymptotic critical values:

\begin{thm}\label{thm_kurdyka}
Let $f:X \to \R^k$ be a ${\cal C}^1$ semi-algebraic mapping, where $X$ is a ${\cal C}^1$-submanifold of $\R^n$. The set of asymptotic critical values has dimension less than $k$. 
\end{thm}
\begin{proof}
We first establish that $K_\infty(f)$ has dimension less than $k$.
It easily follows from Curve Selection Lemma that  we have $$K_\infty(f)=\cup_{i=1} ^\infty C_i,$$ where $C_i$ is the semi-algebraic set defined by 
$$C_i:=\{y \in \R^k|\exists (x_l)_{l\ge l_0}, x_l  \in  X, |x_l|=l, \lim f(x_l) = y,  \nu(d_{x_l} f)\le l^{-(1+1/i)} \} .$$
By the Baire Property, it suffices to establish that $C_i$ has empty interior for all $i$. For this purpose, let for $t \in \R$, $ Y_{t}:=\{x \in B^n(0,1)|\frac{x}{t} \in X\}$ and, for $x\in Y_t$,  $g_t(x):=f(\frac{x}{t})$. 
 Let also
$$Z_{i,t} := \{x \in Y_t\,|\;\nu (d_x g_t)\le t^{1/i} \}.$$ Observe that $Z_{i,0_+}=c_{t^{\frac{1}{i}}}(g_{0_+})$, so that, by Theorem \ref{sard},  $g_{0_+}(Z_{i,0_+})$ is $(k,v_{0_+})$-thin. By Lemma \ref{lem_familles}, this implies that the set\begin{equation}\label{eq_points_critiques}
cl( \bigcup_{t\in \R} (g_t(Z_{i,t}) \times \{t\}))\cap (\R^k \times \{0\})\end{equation}
has dimension less than $k$. As this set contains $C_i$, this yields $\dim C_i<k$, as required.

 It remains to show that $K_1(f)$ has dimension less than $k$. We shall make use of a similar argument. It again easily follows from Curve Selection Lemma  that  $$K_1(f)=\cup_{i =1}^\infty C_i',$$ where $C_i'$ is the semi-algebraic set defined by 
$$C_i':=\{y \in \R^k|\exists (x_l)_{l\ge l_0} \mbox{ such that } x_l\in  X, d(x_l,fr( X))=\frac{1}{l}, \lim f(x_l) = y,  \nu(d_{x_l} f)\le l^{\frac{-1}{i}} \} .$$
Again, by the Baire Property, it suffices to establish that $C_i'$ has empty interior for all $i$.  For $t>0$, tet  $h_t$ be the restriction of $f$ to the set $Y'_{t}:=\{x \in X| \frac{t}{2}<d(x,fr( X))< 2t\}$. By Theorem \ref{sard}, $h_{0_+}(c_{t^{\frac{1}{i}}}(h_{0_+}))$ is $(k,v_{0_+})$-thin for all $i$.
 By Lemma \ref{lem_familles} this establishes that $\dim C'_i<k$  (with the same argument as for $C_i$ and $c_{t^{\frac{1}{i}}}(g_{0_+})$, see (\ref{eq_points_critiques})).
\end{proof}



\begin{rem}\label{rem_fin}These theorems remain valid on an arbitrary o-minimal structure with minor or no modifications of the proofs. Namely, if the o-minimal structure is polynomially bounded then the same argument applies. If it is not assumed to be polynomially bounded, only the proof of Theorem \ref{thm_kurdyka}  requires a slight modification.
 \end{rem}


\begin{thebibliography}{mmm}
\bibitem[BCR]{bcr}J. Bochnak, M. Coste, and M.-F. Roy, G\'eom\'etrie alg\'ebrique r\'eelle, Springer 1987.
 \bibitem[F]{f} A. Fischer, O-minimal Λm-regular stratification,
Ann. Pure Appl. Logic 147 (2007), no. 1-2, 101--112. 
\bibitem[H]{h}R. Hardt, Semi-algebraic local-triviality in semi-algebraic mappings.  Amer. J. Math.  102  (1980), no. 2, 291--302.
\bibitem[K]{k}K. Kurdyka, On a subanalytic stratification satisfying a Whitney property with
exponent 1, Real algebraic geometry, Proc., Rennes 1991, L.N.M. 1524 (1992) 316--322.
\bibitem[KOS]{kos}K. Kurdyka, P. Orro \& S. Simon, Semialgebraic Sard Theorem for generalized critical values, J. Differential Geometry 56 (2000) 67-92.
\bibitem[LR]{lr}J.-M. Lion, J.-P. Rolin, Th\'eor\`eme de pr\'eparation pour les fonctions
logarithmico-exponentielles.
Annales de l'Institut Fourier, tome 47, no 3 (1997), p. 859--884.

\bibitem[R]{r}  P. J. Rabier, Ehresmann’s fibrations and Palais-Smale conditions for morphisms
of Finsler manifolds, Ann. of Math. 146 (1997) 647--691.
\bibitem[V]{v}G. Valette, Vanising homology, Sel. Math. New Series 16 (2010), 267--296. 
\end{thebibliography}
\end{document}